\documentclass[11pt, reqno, draft]{amsart}

\usepackage{mathrsfs}
\usepackage{amsfonts}
\usepackage{amsthm}
\usepackage{amsmath}
\usepackage{graphicx}
\usepackage{amscd,amssymb,amsthm, color}
\usepackage{setspace}

\newcommand\mycom[2]{\genfrac{}{}{0pt}{}{#1}{#2}}
\numberwithin{equation}{section}

\newtheorem{theorem}{Theorem}

\newtheorem{remark}{Remark}

\begin{document}

\allowdisplaybreaks

\thispagestyle{plain}

\vspace{4cc}
\begin{center}

{\LARGE \bf $p$--extended Mathieu series from the Schl\"omilch series point of view}
\rule{0mm}{6mm}\renewcommand{\thefootnote}{} \vspace{1cc}

{\large \bf Dragana Jankov Ma\v sirevi\' c$^\dag$ and Tibor K. Pog\'any$^\ddag$}
\bigskip

{\it $^\dag$ Department of Mathematics, University of Osijek, Trg Lj. Gaja 6, 31000 Osijek, Croatia\\

$^\ddag$ Faculty od Maritime Studies, University of Rijeka,
Studentska 2, 51000 Rijeka, Croatia \and Institute of Applied
Mathematics, \'Obuda University, 1034 Budapest, Hungary}

\footnotetext{E-mail: \texttt{djankov@mathos.hr} (D. Jankov Ma\v
sirevi\' c), \texttt{poganj@pfri.hr} (T. K. Pog\'any)}

\vspace{2cc}

\parbox{24cc}
{\small {\bf Abstract}. Motivated by certain current results by Parmar and Pog\'any \cite{PP} in which the authors introduced the
so--called $p$--extended Mathieu series the main aim of this paper is to present a connection between such series and a various types
of Schl\"omilch series.

\bigskip

\bigskip

}
\end{center}

\section{Introduction and motivation}

The series of the form
   \begin{equation*}
      S(r) = \sum_{n \ge 1}\dfrac{2n}{(n^2+r^2)^2},\qquad r>0
   \end{equation*}
is known in literature as Mathieu series. \'Emile Leonard Mathieu was the first who investigated such series in 1890 in his book
\cite{Mathieu}. There is a wide range of various generalizations of the Mathieu series, and one of them is the so--called
generalized Mathieu series with a fractional power reads \cite[p. 2, Eq. (1.6)]{CL} (and also consult \cite[p. 181]{GP})
   \begin{equation*}
      S_\mu(r) = \sum_{n\ge1}\dfrac{2n}{(n^2+r^2)^{\mu+1}}, \qquad r>0,\,\,\mu>0;
   \end{equation*}
which can also be presented in terms of the Riemann Zeta function \cite[p. 3, Eq. (2.1)]{CL}
   \begin{equation}\label{A21}
      S_{\mu}(r) = 2\sum_{n\ge0} r^{2n}(-1)^n\left({\mu+n} \atop n\right)\zeta(2\mu+2n+1), \qquad |r|<1.
	 \end{equation}
Having in mind \eqref{A21} Parmar and Pog\'any \cite{PP} recently introduced the {\it $p$--extended Mathieu series}
	 \begin{equation}\label{A3}
	    S_{\mu,p}(r) = 2\sum_{n \ge 0} r^{2n} (-1)^{n}\binom{\mu+n}{n} \zeta_{p}(2\mu+2n+1),
   \end{equation}
where $\Re(p)>0$ or $p=0$, $\mu>0$. Here and in what follows $\zeta_{p}$ stands for the $p$--extension of the Riemann $\zeta$ function
\cite{Ch-Qa-Bo-Ra-Zu}:
   \begin{equation*} 
      \zeta_p(\alpha) = \frac1{\Gamma(\alpha)}\int_0^\infty \frac{t^{\alpha-1}{\rm e}^{-\frac{p}{t}}}{{\rm e}^t-1} \,{\rm d}t
   \end{equation*}
defined for $\Re(p)>0$ or $p=0$ and $\Re(\alpha)>0$ and it reduces to the Riemann zeta function when $p=0$. Also, \eqref{A3} one
reduces to \eqref{A21} when $p=0$.

Parmar and Pog\'any \cite{PP} obtained an integral form of such series, which reads
 \begin{equation}\label{A4}
      S_{\mu,p}(r) = \frac{\sqrt{\pi}}{(2r)^{\mu-\frac12}\Gamma(\mu+1)}\int_0^\infty \frac{t^{\mu+\frac12}
			               {\rm e}^{-\frac{p}{t}}}{{\rm e}^t-1} J_{\mu-\frac{1}{2}}(rt)\,{\rm d}t;
   \end{equation}
here $\Re(p)>0$ or $p=0$, $\mu>0$.

On the other hand, the series of the form
   \[ \sum_{n \geq 0} a_n \mathscr B_\nu(n z)\,, \qquad z \in \mathbb C, \]
which building blocks $\mathscr B_\nu(\cdot)$ consist from either Bessel, modified Bessel of the first and second kind
or alike functions such as Hankel, Lommel, Struve, modified Struve and associated ones of a generic order $\nu$, we usually consider under
the common name {\it Schl\"omilch series} \cite{JP, Wat}.

Motivated by that newly introduced Mathieu series which members contain the extension of the Riemann zeta function $\zeta_{p}$ and
also the fact that $\zeta_{p}$ can be presented as Schl\"omilch series of modified Bessel functions of the second kind i.e. as
\cite[p. 1240]{Ch-Qa-Bo-Ra-Zu}
   \[\zeta_p(\alpha) = \dfrac{2\,p^{\frac\alpha2}}{\Gamma(\alpha)} \sum_{n \geq 1} \dfrac{K_\alpha(2\sqrt{np})}
	                     {n^{\frac\alpha2}},\qquad \alpha, p>0\]
our main aim in this paper is to derive new representations of our series in terms of the  various Schl\"omilch series. In the
next section we would derive new representations  of \eqref{A3} in terms of Schl\"omilch series which members contain derivation
(ordinary or fractional) of a combination of Bessel function of the first kind $J_\nu$ and modified Bessel function of the second kind
$K_\nu$. In the last section we would also derive some connection formulas between our Mathieu series and Schl\"omilch series but
this time with members containing only modified Bessel functions of the second kind.

\section{Connection between $S_{\mu, p}(r)$ and Schl\"omilch series of $J_\nu \cdot K_\mu$}

In this section, our main aim is to derive connection formulas between $p$--extended Mathieu series $S_{\mu,p}(r)$ and Schl\"omilch
series which members contain combination of Bessel functions of the first kind $J_\nu$ and modified Bessel functions of the second
kind $K_\nu$ of the order $\nu$.

Our derivation procedure requires the Gr\"unwald--Letnikov fractional derivative of order $-\alpha$, $\alpha > 0$
with respect to an argument $x$ of a suitable function $f$ defined by \cite{Samko}
   \begin{equation} \label{GL}
      \mathbb D_x^{-\alpha} [f] = \lim_{n \to \infty} \left( \frac {\displaystyle n}{\displaystyle x - a} \right)^\alpha
                                  \sum_{m = 0}^n \frac{\Gamma (\alpha + m)}{m!\, \Gamma (\alpha)}
                                  f\left(x - m \frac{x - a}h\right), \qquad a <x\,.
   \end{equation}
Several numerical algorithms are available for the direct computation of \eqref{GL}; see e.g. \cite{Diet, Murio, Sousa}.

\begin{theorem}
For all $\min\{\Re(p),\Re(q),\gamma\} >0$ and $\alpha>\tfrac12$  there holds
   \begin{align}\label{Al1}
      S_{\alpha-\frac32, p}(\gamma) &= \dfrac{2(-1)^\alpha\sqrt{\pi}}{(2\gamma)^{\alpha-2}\Gamma(\alpha-\frac12)}
                        \sum_{k \geq 1} \mathbb{D}_q^\alpha \Big( J_{\alpha-2}
												\left(\sqrt{2p}\, [\sqrt{q^2+\gamma^2}-q]^{\frac12}\right) \\ \nonumber
                &\qquad \times K_{\alpha-2}\left(\sqrt{2p}\, [\sqrt{q^2+\gamma^2}+q]^{\frac12}\right) \Big) \Big|_{q=k} \, .
   \end{align}
Further, for $\alpha=n\in\mathbb{N}$ we have
   \begin{align*}
      S_{n-\frac32, p}(\gamma) &= \dfrac{2(-1)^n\sqrt{\pi}}{(2\gamma)^{n-2}\Gamma(n-\frac12)}
                   \sum_{k \geq 1} \frac{\partial^n}{\partial q^n} \Big[J_{n-2}
									 \left(\sqrt{2p}\,[\sqrt{q^2+\gamma^2}-q]^{\frac12}\right)\nonumber \\
           &\qquad \times K_{n-2}\left(\sqrt{2p}\,[\sqrt{q^2+\gamma^2}+q]^{\frac12}\right)\Big]\Big|_{q=k}\, \, .
   \end{align*}
\end{theorem}

\begin{proof}
In order to prove the desired results, let us first consider the integral \cite[p. 708, Eq. 6.635.3]{GR}
   \begin{align} \label{Int}
        A_{p,q}(\gamma) &= \int_0^\infty x^{-1} {\rm e}^{-qx-p/x} J_\nu(\gamma x) {\rm d}x\\ \nonumber
                            &= 2\, J_\nu\left(\sqrt{2p}\,[\sqrt{q^2+\gamma^2}-q]^{\frac12}\right)\,
                               K_\nu\left(\sqrt{2p}\,[\sqrt{q^2+\gamma^2}+q]^{\frac12}\right)\,,
   \end{align}
where $\min\{\Re(p),\Re(q),\gamma\} >0$.\footnote{Actually, $A_{p,q}(\gamma)$ is the Laplace transform of $x \mapsto x^{-1}
{\rm e}^{-p/x} J_\nu(\gamma x)$ at the argument $q$.}

Now, using the Gr\"unwald-Letnikov fractional derivative
   \begin{equation*}
      \mathbb D_q^\alpha {\rm{e}}^{-q x} = (-x)^\alpha\,{\rm{e}}^{-qx}
   \end{equation*}
valid for every real $\alpha> -\nu $ we get
   \[ \mathbb{D}_q^\alpha  A_{p,q}(\gamma)  = (-1)^\alpha\int_0^\infty x^{\alpha-1} {\rm
e}^{-qx-p/x} J_\nu(\gamma x) {\rm d}x. \] Further, specifying $q
=k+1$ and summing up the previous equality for $k \in \mathbb N_0$
we have
   \[ \sum_{k \geq 0} \mathbb{D}_q^\alpha A_{p,q}(\gamma)\big|_{q=k+1}
           = (-1)^\alpha \int_0^\infty \frac{x^{\alpha-1} {\rm e}^{-p/x}}{{\rm e}^x-1} J_\nu(\gamma x) {\rm d}x\, .\]
Setting $\nu=\alpha-2$ with the help of
the integral representation \eqref{A4}  we get
   \begin{align*}
       \int_0^\infty \frac{x^{\alpha-1} {\rm e}^{-p/x}}{{\rm e}^x-1} J_{\alpha-2}(\gamma x) {\rm d}x
                 &= \dfrac{(2\gamma)^{\alpha-2}\Gamma(\alpha-\frac12)}{\sqrt{\pi}}S_{\alpha-\frac32, p}(\gamma)
   \end{align*}
Now, from the previous calculations, using also  \eqref{Int}, we have
   \begin{align*}
      S_{\alpha-\frac32, p}(\gamma) &= \dfrac{(-1)^\alpha\sqrt{\pi}}{(2\gamma)^{\alpha-2}\Gamma(\alpha-\frac12)}
                                   \sum_{k \geq 1} \mathbb{D}_q^\alpha A_{p,q}(\gamma) \Big|_{q=k} \\
                    &= \dfrac{2(-1)^\alpha\sqrt{\pi}}{(2\gamma)^{\alpha-2}\Gamma(\alpha-\frac12)}
                             \sum_{k\geq 1} \mathbb{D}_q^\alpha \Big(J_{\alpha-2}\left(\sqrt{2p} [\sqrt{q^2+\gamma^2}-q]^{\frac12}\right) \\
                    &\qquad\times K_{\alpha-2}\left(\sqrt{2p} [\sqrt{q^2+\gamma^2}+q]^{\frac12}\right) \Big) \Big|_{q=k} \,  \bigskip
   \end{align*}
which is equal to \eqref{Al1}.

Next, for a positive integer $\alpha=n$ (in fact $A_{p,q}(\gamma)$
converges for all $n+\nu>0$) consider
   \[ \frac{\partial^n}{\partial q^n}\, A_{p,q}(\gamma) = (-1)^n\int_0^\infty x^{n-1} {\rm e}^{-qx-p/x} J_\nu(\gamma x) {\rm d}x. \]
The same procedure as above yields
   \[ \int_0^\infty \frac{x^{n-1} {\rm e}^{-p/x}}{{\rm e}^x-1} J_\nu(\gamma x) {\rm d}x
            = (-1)^n\, \sum_{k \geq 0} \frac{\partial^n}{\partial q^n}A_{p,q}(\gamma)\Big|_{q=k+1} \,.\]
Again with the help of \eqref{A4} and \eqref{Int} and substituting $\nu=n-2$ we have
   \begin{align*}
      S_{n-\frac32, p}&(\gamma) = \dfrac{(-1)^n\sqrt{\pi}}{(2\gamma)^{n-2}\Gamma(n-\frac12)}\,
                                  \sum_{k \geq 1} \frac{\partial^n}{\partial q^n} A_{p,q}(\gamma)\Big|_{q=k}
                                = \dfrac{2(-1)^n\sqrt{\pi}}{(2\gamma)^{n-2}\Gamma(n-\frac12)}\\
                   &\qquad \times \sum_{k \geq 1} \frac{\partial^n}{\partial q^n}
                                  \left[J_{n-2}\left(\sqrt{2p}[\sqrt{q^2+\gamma^2}-q]^{\frac12}\right)\,
                                  K_{n-2}\left(\sqrt{2p}[\sqrt{q^2+\gamma^2}+q]^{\frac12}\right)\right]\Big|_{q=k}\, ,
   \end{align*}
which completes the proof.
\end{proof}

\begin{theorem} For all $\Re(p)>0$ we have
   \begin{align}\label{B1}
      S_{\frac12,p}(\gamma)& = -4\gamma \sum_{k\ge1} \frac{\partial^3}{\partial q^3}\,
			                     \left(\dfrac{J_1(\sqrt{2p}[\sqrt{q^2+\gamma^2} - q]^{\frac12})K_0(\sqrt{2p}[\sqrt{q^2+\gamma^2}
												   + q]^{\frac12})}{\sqrt{2p}[\sqrt{q^2+\gamma^2} + q]^{\frac12}}\right.\\ \nonumber
											 &\qquad+\left. \dfrac{J_0(\sqrt{2p}[\sqrt{q^2+\gamma^2} - q]^{\frac12})
											     K_1(\sqrt{2p}[\sqrt{q^2+\gamma^2} + q]^{\frac12})}{\sqrt{2p}[\sqrt{q^2+\gamma^2}
													 - q]^{\frac12}}\right)\Big|_{q=k}.
	 \end{align}
Moreover, it is
   \begin{equation}\label{B2}
      S_{-\frac12,p}(\gamma) = 4\gamma \sum_{k\ge1} \frac{\partial}{\partial q}\,
			                     \left(J_1(\sqrt{2p}[\sqrt{q^2+\gamma^2} - q]^{\frac12}) K_1(\sqrt{2p}[\sqrt{q^2+\gamma^2}
													 + q]^{\frac12})\right)\Big|_{q=k}.
	 \end{equation}
\end{theorem}

\begin{proof}
With the help of the  integral \cite[p. 188, Eq. {\bf{2.12.10.2}}]{Prudnikov2}
   \[ B_{p,q}(\gamma)=\int_0^\infty x^{-2}{\rm e}^{-qx-p/x}J_0(\gamma x)\,{\rm d}x
           = 2\gamma \left(z_+^{-1}J_1(z_-)K_0(z_+)+z_-^{-1}J_0(z_-)K_1(z_+)\right),\]
where $z_{\pm}=\sqrt{2p}[\sqrt{q^2+\gamma^2}\pm q]^{1/2}$, $\min\{\Re(q),\Re(p)\}>0$, we conclude
   \[ \frac{\partial^3}{\partial q^3}\, B_{p,q}(\gamma) = -\int_0^\infty x {\rm e}^{-qx-p/x} J_0(\gamma x) {\rm d}x \]
which, with the help of \eqref{A4}, gives us
   \[ \sum_{k \geq 0}\frac{\partial^3}{\partial q^3}\, B_{p,q}(\gamma)\Big|_{q=k+1}
	          = -\int_0^\infty \frac{x {\rm e}^{-p/x}}{{\rm e}^x-1} J_0(\gamma x){\rm d}x = -\dfrac12\,S_{1/2,p}(\gamma) \]
which coincides with \eqref{B1}.

In the same way, but this time using \cite[p. 188, Eq. {\bf{2.12.10.1}}]{Prudnikov2}
   \[ C_{p,q}(\gamma)=\int_0^\infty x^{-1}{\rm e}^{-qx-p/x}J_\nu(\gamma x)\,{\rm d}x
           = 2J_\nu(z_-)K_\nu(z_+),\]
where $\min\{\Re(p),\Re(q)\}>0$, and $z_{\pm}$ has the same meaning as above, with the aid of parity of Bessel and modified
Bessel function $J_{-1}(x)=-J_1(x); K_{-1}(x)=K_1(x)$ we deduce \eqref{B2}.
\end{proof}

\begin{remark}
From \eqref{B2}, bearing in mind \cite{Wolf1, Wolf2}:
\[2\left(J_1(x)\,K_1(x)\right)'=\left(J_0(x)-J_2(x)\right)K_1(x)-J_1(x)\left(K_0(x)+K_2(x)\right),\]
we can infer a new representation for $S_{-\frac12,p}(\gamma)$.
\end{remark}

\section{$S_{\mu, p}(r)$ and the Schl\"omilch series of $K_\nu$ terms}

Considering now specialized $p$--extended Mathieu series, that is in which $\mu = 0, 1, 2$, we report on their Schl\"omilch--series 
expansion {\em via} modified Bessel functions of the second kind $K_{\mu+1}$. 

\begin{theorem}\label{tm3}
For all $\Re(p)>0$, $\gamma>0$ there hold
   \begin{align} \label{B3}
      S_{0,p}(\gamma) & = 2\sqrt{p}\,\sum_{k\ge1} \left(\dfrac{K_1\left(2\sqrt{p(k+{\rm i}\,\gamma)}\right)}{\sqrt{k+{\rm i}\,\gamma}}
			                  + \dfrac{K_1\left(2\sqrt{p(k-{\rm i}\,\gamma)}\right)}{\sqrt{k-{\rm i}\,\gamma}}\right), \\ \label{B4}
      S_{1,p}(\gamma) & = \dfrac{p\,{\rm{i}}}{\gamma}\sum_{k\ge1}\left(\dfrac{K_2\left(2\sqrt{p(k+{\rm i}\,\gamma)}\right)}
			                    {\sqrt{k+{\rm i}\,\gamma}}-\dfrac{K_2\left(2\sqrt{p(k-{\rm i}\,\gamma)}\right)}
													{\sqrt{k-{\rm i}\,\gamma}}\right).
   \end{align}
\end{theorem}

\begin{proof}
In order to prove the desired results we will need the following formula \cite{Prudnikov4}
   \begin{align}\label{B5}
      E_{p,q}^{\mp}(\gamma) &= \int_0^\infty x^\nu {\rm e}^{-qx-p/x} \left\{\mycom{\sin(\gamma x)}
			                         {\cos(\gamma x)}\right\}\,{\rm d}x\\\nonumber
														&= {\rm{i}}^{\frac{1\pm1}{2}}p^{\frac{\nu+1}{2}}\left(\dfrac{K_{\nu+1}
														   \left(2\sqrt{p(q+{\rm i}\,\gamma)}\right)}{(q+{\rm i}\,\gamma)^{\frac{\nu+1}{2}}}
													\mp \dfrac{K_{\nu+1}\left(2\sqrt{p(q-{\rm i}\,\gamma)}\right)}{(q-{\rm i}\,\gamma)^{\frac{\nu+1}{2}}}\right),
	 \end{align}
which holds for $\min\{\Re(p),\Re(q)\}>0$.

Now, since
   \[J_{-\frac12}(x) = \sqrt{\dfrac{2}{\pi x}}\,\cos x,\]
by virtue of \eqref{A4} and \eqref{B5} setting $q=k+1$, $k\in\mathbb{N}_0$ and $\nu=0$ it follows
   \[\sum_{k \geq 0} E^+_{p,k+1}(\gamma) = \sqrt{\dfrac{\pi\gamma}{2}}\int_0^\infty\dfrac{\sqrt{x}\,
	                 {\rm{e}}^{-p/x}}{{\rm{e}}^x-1}\,J_{-\frac12}(\gamma x)\,{\rm{d}}x = \dfrac12 S_{0,p}(\gamma),\]
which results in \eqref{B3}.

Analogously, from \eqref{A4} for $\nu=1$, applying \eqref{B5} for $E^-_{p,k+1}(\gamma)$ and $J_{\frac12}(x) = \sqrt{2/(\pi x)}\, \sin x$,
one implies the second statement \eqref{B4}.
\end{proof}

\begin{theorem}\label{tm4}
For all $\Re(p)>0$, $r>0$ there holds
   \begin{equation}\label{B7}
      S_{2,p}(r) = \dfrac1{(2r)^2}S_{1,p}(r)-\dfrac{p^{\frac32}}{(2r)^2}\sum_{n\ge1}
			             \left(\dfrac{K_3\left(2\sqrt{p(n+{\rm i}\,r)}\right)}{(n+{\rm i}\,r)^{\frac32}}
								 + \dfrac{K_3\left(2\sqrt{p(n-{\rm i}\,r)}\right)}{(n-{\rm i}\,r)^{\frac32}}\right).
   \end{equation}
\end{theorem}
\begin{proof}
From the integral representation \eqref{A4} of $S_{\mu,p}(r)$, for $\mu=2$ it is
   \begin{align*}
      S_{2,p}(r) &= \dfrac{\sqrt{\pi}}{2(2r)^{3/2}}\int_0^\infty\dfrac{x^{5/2}{\rm e}^{-p/x}J_{\frac32}(rx)}{{\rm e}^x-1}\, {\rm d}x \\
                 &= \dfrac{\sqrt{\pi}}{2(2r)^{3/2}}\sum_{k\ge1}\int_0^\infty x^{5/2}{\rm e}^{-kx-p/x}J_{\frac32}(rx)\,{\rm d}x\\
								 &= \dfrac{1}{(2r)^2}\sum_{k\ge1}\int_0^\infty x^2 {\rm e}^{-kx-p/x}\left(\dfrac{\sin(rx)}{rx}-\cos(rx)\right)\,{\rm d}x,
   \end{align*}
where in the last equality we used the well--known formula
   \[J_{\frac32}(x) = \sqrt{\dfrac{2}{\pi x}}\left(\dfrac{\sin x}{x}-\cos x\right). \]
Further, with the help of \eqref{A4} and $J_{\frac12}(x) = \sqrt{2/(\pi x)}\, \sin x$ the previous expression can be rewritten into
   \[S_{2,p}(r) = \dfrac1{(2r)^2}\, S_{1,p}(r) - \dfrac1{(2r)^2}\sum_{k\ge1}\int_0^\infty x^{2}
	                {\rm e}^{-kx-p/x}\cos(rx)\,{\rm d}x. \]
Finally, using the Laplace transform of the function $x \mapsto x^2{\rm e}^{-p/x}\cos(rx)$, in the argument $k$ given by \eqref{B5} we
get the display \eqref{B7}.
\end{proof}

\begin{remark}
Using the formula \eqref{B4} derived in {\rm Theorem~\ref{tm3}} and the formula \eqref{B7} which connects $S_{2,p}(r)$ and $S_{1,p}(r)$ new representation for $S_{2,p}(r)$ can be derived.
\end{remark}

\section*{Acknowledgements}
This paper has been partially supported by Grant No. HRZZ-5435.

\end{document}